\documentclass{llncs}

\usepackage{amssymb,stmaryrd,mathrsfs,dsfont,mathtools}

\input xy
\xyoption{all}
\usepackage{xfrac} 
\usepackage{tikz}
\tikzstyle{vertex}=[circle, draw, inner sep=0pt, minimum size=3pt]
\newcommand{\vertex}{\node[vertex]}

\begin{document}
\frontmatter          
\pagestyle{headings}  

\mainmatter              

\title{Characterization of Word-Representable Graphs using Modular Decomposition}

\titlerunning{Title}  
%

\author{Tithi Dwary \and 
K. V. Krishna} 

\authorrunning{Tithi Dwary \and K. V. Krishna} 

\institute{Indian Institute of Technology Guwahati, India\\
	\email{tithi.dwary@iitg.ac.in};\;\;\; 
	\email{kvk@iitg.ac.in}}

\maketitle              

\begin{abstract}
In this work, we characterize the class of word-representable graphs with respect to the modular decomposition. Consequently, we determine the representation number of a word-representable graph in terms of the permutation-representation numbers of the modules and the representation number of the associated quotient graph. In this connection, we also obtain a complete answer to the open problem posed by Kitaev and Lozin on the word-representability of the lexicographical product of graphs.
\end{abstract}

\keywords{Word-representable graphs, representation number, comparability graphs, lexicographical product, modular decomposition.}

\section{Introduction and Preliminaries}

A word  over a finite set of letters, say $A$, is a finite sequence of letters of $A$ written by juxtaposing them. A subsequence $u$ of the sequence of a word $w$ is said to be a subword of the word $w$ and it is denoted by $u \ll w$. A subword of the word $w$ over a subset $B \subseteq A$, denoted by $w_B$, is obtained from $w$ by deleting all the occurrences of the letters of $w$ belonging to $A \setminus B$. 

In this paper, we consider only simple (i.e., without loops or parallel edges) and connected graphs. A graph $G = (V, E)$ is said to be a word-representable graph if there exists a word $w$ over its vertex set $V$ such that for all $a, b \in V$, $a$ and $b$ are adjacent in $G$ if and only if $w_{\{a, b\}}$ is either of the form $ababa\cdots$ or $babab\cdots$, which can be of even or odd length. In this case, we say $a$ and $b$ alternate in $w$. The class of word-representable graphs includes several fundamental classes of graphs such as comparability graphs, circle graphs, parity graphs, and 3-colorable graphs. Further, a word-representable graph is said to be a $k$-word-representable graph if there exists a word $w$ over its vertex set such that every letter appears exactly $k$ times in $w$; in this case, we say $w$ is a $k$-uniform word. A $1$-uniform word $w$ is a permutation on the set of letters of $w$. It is known that a graph is word-representable if and only if it is $k$-word-representable for some $k$ \cite{MR2467435}. The representation number of a word-representable graph $G$ is the smallest $k$ such that $G$ is $k$-word-representable, and it is denoted by $\mathcal{R}(G)$. Further, it is known that the general problem of determining the representation number of a word-representable graph is computationally hard \cite{Hallsorsson_2011}. For a detailed introduction to the theory of word-representable graphs, one may refer to the monograph by Kitaev and Lozin \cite{words&graphs}. 

A word-representable graph $G$ is said to be permutationally representable if there exists a word $w$ of the form $p_1p_2 \cdots p_k$ representing the graph $G$, where each $p_i$, $1 \le i \le k$, is a permutation of the vertices of $G$. In this case, $G$ is said to be a permutationally $k$-representable graph. The smallest $k$ such that $G$ is a permutationally $k$-representable graph is said to be the permutation-representation number (in short \textit{prn}) of $G$, and it is denoted by $\mathcal{R}^p(G)$. Note that for a permutationally representable graph $G$, $\mathcal{R}(G) \le \mathcal{R}^p(G)$.

To distinguish between edges of graphs and two-letter words, we write $\overline{ab}$ to denote an undirected edge between vertices $a$ and $b$. A directed edge from $a$ to $b$ is denoted by $\overrightarrow{ab}$. An orientation of a graph assigns a direction to each edge so that the obtained graph is a directed graph. A graph $G = (V, E)$ is said to be a comparability graph if it admits a transitive orientation, i.e., an orientation in which if $\overrightarrow{ab}$ and $\overrightarrow{bc}$, then $\overrightarrow{ac}$, for all $a, b, c \in V$. In \cite{perkinsemigroup}, it was proved that a graph is a comparability graph if and only if it is permutationally representable. Hence, the class of comparability graphs is a subclass of the class of word-representable graphs. Further, the class of word-representable graphs is characterized in terms of semi-transitive orientations, a generalization of transitive orientations \cite{Halldorsson_2016}.

The class of comparability graphs is a well studied graph class in the literature and it receives importance due to its connection with posets. Every comparability graph $G$ induces a poset, say $P_G$, based on its transitive orientation. Moreover, it is known that for a comparability graph $G$, $\mathcal{R}^p(G) = k$ if and only if $\dim(P_G) = k$ (cf. \cite{khyodeno2}). Further, it was proved that the problem of deciding whether a poset has dimension at most $k$, for $k \ge 3$, is NP-complete \cite{yannakakis1982complexity}. Accordingly, determining the \textit{prn} of a comparability graph is computationally hard.

Let $G = (V, E)$ be a graph. The neighborhood of a vertex $a \in V$ is denoted by $N_G(a)$, and is defined by $N_G(a) = \{b \in V \mid \overline{ab}\in E\}$. Further, the subgraph of $G$ induced by $A \subseteq V$ is denoted by $G[A]$. A set $M \subseteq V$ is a module of $G$ if, for any $a \in V \setminus M$, either $N_G(a) \cap M = \varnothing$ or $M \subseteq N_G(a)$. The singletons and the whole set $V$ are the trivial modules of $G$.  A graph is said to be a prime graph if it contains only trivial modules, otherwise it is said to be decomposable. We say two modules $M$ and $M'$ overlap if $M \cap M' \neq \varnothing$,  $M \setminus M' \neq \varnothing$ and  $M' \setminus M \neq \varnothing$. A module $M$ is said to be strong if it does not overlap with any other module of $G$. Further, a module $M$ is called maximal if there is no module $M'$ of $G$ such that $M \subset M' \subset V$. 
Note that any two disjoint modules $M$ and $M'$ are either adjacent, i.e., any vertex of $M$ is adjacent to all vertices of $M'$, or non-adjacent, i.e., no vertex of $M$ is adjacent to any vertex of $M'$. 

Let $\mathcal{P} = \{M_1, M_2, \ldots, M_k\}$ be a partition of the vertex set of a graph $G$ such that each $M_i$ is a module of $G$. Then $\mathcal{P}$ is said to be a modular partition of $G$. A quotient graph associated to $\mathcal{P}$, denoted by $\sfrac{G}{\mathcal{P}}$, is a graph whose vertices are in a bijection with the elements of $\mathcal{P}$ and two vertices in $\sfrac{G}{\mathcal{P}}$ are adjacent if and only if the corresponding modules are adjacent in $G$. It is evident that $\sfrac{G}{\mathcal{P}}$ is isomorphic to an induced subgraph of $G$. A modular partition of $G$ containing only maximal strong modules is called a maximal modular partition. It is known that every graph has a unique maximal modular partition (cf. \cite{habib2010survey}).

\begin{remark}\label{recons_mod_dec}
	Let $\mathcal{P} = \{M_1, M_2, \ldots, M_k\}$ be a modular partition of a graph $G$. Then, the original graph $G$ can be reconstructed from $\sfrac{G}{\mathcal{P}}$ by replacing each vertex of $\sfrac{G}{\mathcal{P}}$ with the corresponding modules $G[M_i]$, for $1 \le i \le k$.
\end{remark}

The modular decomposition of a graph was first introduced by Gallai \cite{Gallaipaper} to study the structure of comparability graphs. Further, it was used to characterize a number of graph classes viz., cographs, permutation graphs, interval graphs (cf. \cite{Golumbicbook_2004}). Other than graph theory, it has a large range of applications in the theory of posets, scheduling problems. For a detail information on this topic, one may refer to \cite{habib2010survey,Mohring_1985}. A modular decomposition of a graph can be computed in linear time \cite{McConnell_1994}. In the following, we present the structural characterization of comparability graphs with respect to the modular decomposition. 

\begin{theorem}[\cite{Golumbicbook_2004,Mohring_1985}] \label{mod_com}
	Let $\mathcal{P} = \{M_1, M_2, \ldots, M_k\}$ be a modular partition of a decomposable graph $G$. Then $G$ is a comparability graph if and only if $\sfrac{G}{\mathcal{P}}$ and each of the induced subgraphs $G[M_i]$ are comparability graphs.
\end{theorem}

It is evident that the modular decomposition is useful for the recognition of comparability graphs. We state the following problem for word-representable graphs.

\begin{problem}
	Characterize word-representable graphs with respect to the modular decomposition.
\end{problem}

In this paper, we address Problem 1 and as a consequence, we also determine the representation number of a word-representable graph in terms of the permutation-representation numbers of its modules and the representation number of the quotient graph. Additionally, we provide a necessary and sufficient condition under which the lexicographical product of two word-representable graphs is a word-representable graph. This gives a complete answer to the open problem posed in \cite[Chapter 7]{words&graphs}.

\section{Characterization}

In the present context, to study modular decomposition, first we consider the following operation which replaces a module in place of a vertex in a graph. Let $G = (V \cup \{a\}, E)$ and $M = (V', E')$ be two graphs. The graph $G_a[M] = (V'', E'')$ is defined by $V'' = V \cup V'$ and the edge set $E''$ consists of the edges of $G[V]$, edges of $M$ and $\{\overline{bc} \mid b \in V', c \in N_G(a)\}$. Note that $V'$ is a module in the graph $G_a[M]$. We say $G_a[M]$ is obtained from $G$ by replacing the vertex $a$ of $G$ with the module $M$.

\begin{remark}
	For two word-representable graphs $G$ and $M$, $G_a[M]$ is not always a word-representable graph. For instance, consider the graph obtained by replacing one of the two vertices of the complete graph $K_2$ with the cycle $C_5$. This graph is nothing but the wheel graph $W_5$ (see Fig. \ref{w_5}), which is not a word-representable graph (cf. \cite[Chapter 3]{words&graphs}). 
\end{remark}

\begin{figure}[ht]
	\centering
	\begin{minipage}{.2\textwidth}
		\centering
		\[\begin{tikzpicture}[scale=0.6]
			\vertex (a_8) at (0,0) [fill=black, label=below:$a$] {};
			\vertex (a_7) at (1.5,0) [fill=black, label=below:$b$] {};
			\path
			(a_7) edge (a_8);	
		\end{tikzpicture}\] 
	      $G$
	\end{minipage}%
\begin{minipage}{.4\textwidth}
	\centering
	
	\[\begin{tikzpicture}[scale=0.6]
		\vertex (a_1) at (0,1) [fill=black,label=above:$1$] {};  
		\vertex (a_2) at (2,-1) [fill=black, label=right:$2$] {};
		\vertex (a_3) at (1,-2.5) [fill=black, label=below:$3$] {};
		\vertex (a_4) at (-1,-2.5) [fill=black, label=below:$4$] {};
		\vertex (a_5) at (-2,-1) [fill=black, label=left:$5$] {};
		
		\path
		(a_1) edge (a_2)
		(a_2) edge (a_3)
		(a_3) edge (a_4)
		(a_4) edge (a_5)
		(a_5) edge (a_1);
	\end{tikzpicture}\]
   $C_5$
\end{minipage}%
	\begin{minipage}{.4\textwidth}
		\centering
		
		\[\begin{tikzpicture}[scale=0.6]
			\vertex (a_1) at (0,1) [fill=black, label=above:$1$] {};  
			\vertex (a_2) at (2,-1) [fill=black, label=right:$2$] {};
			\vertex (a_3) at (1,-2.5) [fill=black, label=below:$3$] {};
			\vertex (a_4) at (-1,-2.5) [fill=black, label=below:$4$] {};
			\vertex (a_5) at (-2,-1) [fill=black, label=left:$5$] {};
			\vertex (a_6) at (0,-1) [fill=black, label=below:$b$] {};
			\path
			(a_1) edge (a_2)
			(a_2) edge (a_3)
			(a_3) edge (a_4)
			(a_4) edge (a_5)
			(a_5) edge (a_1)
			(a_6) edge (a_1)
			(a_6) edge (a_2)
			(a_6) edge (a_3)
			(a_6) edge (a_4)
			(a_6) edge (a_5);
		\end{tikzpicture}\]
	$G_a[C_5]$
	\end{minipage}%
	
	\caption{Replacing the vertex $a$ of $G$ with the module $C_5$} 
	\label{w_5}
\end{figure}
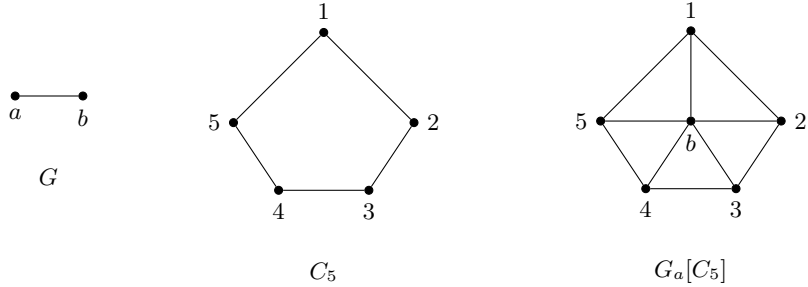

In what follows, $G$ and $M$ denote the graphs $G = (V \cup \{a\}, E)$ and $M = (V', E')$, unless stated otherwise. The following theorem gives us a sufficient condition for the word-representability of $G_a[M]$.  
 
\begin{theorem}[\cite{Kitaev_2013}]\label{Module_0}
	Suppose $G$ is a word-representable graph. If $M$ is a comparability graph, then $G_a[M]$ is word-representable. 
\end{theorem}

Further, on the representation number of $G_a[M]$, it was stated in \cite{Kitaev_2013} that if $\mathcal{R}(G) = k$ and $\mathcal{R}(M) = k'$, then $\mathcal{R}(G_a[M]) = \max\{k, k'\}$. Although it was stated that $k'$ is the representation number of the module $M$, the concept of \textit{prn} of $M$ was used in the proof given in \cite{Kitaev_2013}. If $k' = \mathcal{R}(M)$, then the statement does not hold good, as shown in the following counterexample.      

\begin{example}
	Consider the graph obtained by replacing a vertex, say $a$, of $G = K_2$ with the cycle $C_6$. Note that $\mathcal{R}(G) = 1$ and $\mathcal{R}(C_6) = 2$ \cite[Chapter 3]{words&graphs}. But $\mathcal{R}(G_a[C_6]) \ne 2$. In fact, $G_a[C_6]$ is the wheel $W_6$ and its representation number is 3 (by \cite[Lemma 3]{Hallsorsson_2011}), as $\mathcal{R}^p(M) = 3$ (cf. \cite{mozhui2023}). 
\end{example}

In the following, we state the correct version for giving representation number of $G_a[M]$. 

\begin{theorem}\label{Module_1}
	Let $G$ be a word-representable graph and $M$ be a comparability graph. If $\mathcal{R}(G) = k$ and $\mathcal{R}^p(M) = k'$, then $\mathcal{R}(G_a[M]) = \max\{k, k'\}$. 
\end{theorem}

We prove that the comparability of $M$ is necessary for $G_a[M]$ to be word-representable graph.

\begin{theorem}\label{Module_3}
	 The graph $G_a[M]$ is word-representable if and only if $G$ is a word-representable graph and $M$ is a comparability graph. 
\end{theorem}

\begin{proof}
	If $G$ is a word-representable graph and $M$ is a comparability graph, then by Theorem \ref{Module_0}, $G_a[M]$ is a word-representable graph. Conversely, suppose $G_a[M]$ is a word-representable graph. Then, $G$ is word-representable as $G$ is isomorphic to an induced subgraph of $G_a[M]$. 
	Since $G$ is connected, there exists a vertex, say $b$ of $G$, such that $b \in N_G(a)$. Then, from the construction of $G_a[M]$ it is evident that $b$ is adjacent to all vertices of $M$. Consider the induced subgraph $G_1 = (G_a[M])[V' \cup \{b\}]$. Clearly, $G_1$ is a word-representable graph. Further, note that $N_{G_1}(b)$ is $M$. Then, in view of \cite[Theorem 9]{MR2467435}, $M$ is a comparability graph. \qed
\end{proof}

In the following theorem, we provide a characterization for $G_a[M]$ to be a comparability graph. We also present the \textit{prn} of $G_a[M]$.   

\begin{theorem}\label{Module_4}
	The graph $G_a[M]$ is a comparability graph if and only if $G$ and $M$  are comparability graphs. Moreover, if $\mathcal{R}^p(G) = k$ and $\mathcal{R}^p(M) = k'$, then $\mathcal{R}^p(G_a[M]) = \max\{k, k'\}$.
\end{theorem}

\begin{proof}
	Suppose, $G_a[M]$ is a comparability graph. Since, $G$ and $M$ are isomorphic to certain induced subgraphs of $G_a[M]$, $G$ and $M$  are comparability graphs. Conversely, suppose $G$ and $M$  are comparability graphs.  Let the words $p_1p_2 \cdots p_{k}$ and $p'_1p'_2 \cdots p'_{k'}$ represent the graphs $G$ and $M$ respectively, where each $p_i$ $(1 \le i \le k)$ is a permutation on the vertices of $G$ and each $p'_i$ $(1 \le i \le k')$ is a permutation on the vertices of $M$. 
	
	Suppose $\max\{k, k'\} = t$. If $k' < k$, then set $p'_j = p'_{k'}$ for all $k'+ 1 \le j \le t$ and note that $p'_1p'_2 \cdots p'_t$ represents $M$. Similarly, if $k < k'$, then set $p_j = p_{k}$ for all $k + 1 \le j \le t$ and note that $p_1p_2 \cdots p_t$ represents $G$. In any case, the words $w = p_1p_2 \cdots p_t$ and $w' = p'_1p'_2 \cdots p'_t$ represent the graphs $G$ and $M$, respectively. For $1 \le i \le t$, let $p_i = r_ias_i$ so that $p_{i_V} = r_is_i$.
	
	Let $v_i = r_ip'_is_i$, for all $1 \le i \le t$. Note that each $v_i$ is a permutation on the vertices of $G_a[M]$. We show that the word $v = v_1v_2 \cdots v_t$ represents the graph $G_a[M]$.
	
	Note that $G[V]$ and $M$ are induced subgraphs of $G_a[M]$. Further, since $v_{V} =  p_{1_V}p_{2_V} \cdots p_{t_V} = w_{V}$ and $v_{V'} = w'$, the subwords $v_{V}$ and $v_{V'}$ of $v$ represent the graphs $G[V]$ and $M$, respectively. Thus, any two vertices of $G[V]$ (or any two vertices of $M$) are adjacent if and only if they alternate in the word $v$.
	
	 Let $b, b'$ be two vertices of $G_a[M]$ such that $b \in V$ and $b' \in V'$. Then $b$ and $b'$ are adjacent in  $G_a[M]$ if and only if $b \in N_G(a)$. Further, note that each $v_i$ is constructed from $p_i$ by replacing $a$ with $p'_i$. Then, it is easy to see that for every $b' \in V'$, $b'$ alternates with $b \in V$ in $v$  if and only if $b$ alternates with $a$ in $w$. Thus, for every $b' \in V'$ and $b \in V$, $b'$ alternates with $b$ in $v$ if and only if  $b \in N_G(a)$. Hence, $v$ represents the graph $G_a[M]$ permutationally.
	
	Therefore, we have $\mathcal{R}^p(G_a[M]) \le t = \max\{k, k'\}$. Further, note that $G$ and $M$ are isomorphic to certain induced subgraphs of $G_a[M]$ so that $\mathcal{R}^p(G_a[M]) \geq t$. Hence, $\mathcal{R}^p(G_a[M]) = \max\{k, k'\}$. \qed	
\end{proof}

As a consequence of the above results, we now study the word-representability of the lexicographical product of any two graphs.

Let $G = (V, E)$ and $G' = (V', E')$ be two graphs. The lexicographical product of $G$ and $G'$, denoted by $G[G']$, defined by $G[G'] = (V'', E'')$, where the vertex set $V'' = V \times V'$ and the edge set $E'' = \{\overline{(a, a')(b, b')} \mid \overline{ab} \in E \ \text{or} \ (a = b, \ \overline{a'b'} \in E')\}$. 

The word-representability of lexicographical product of graphs is an open problem posed in \cite[Chapter 7]{words&graphs}. In connection to this problem,  it was shown in \cite[Section 6]{choi2019operations} that the class of word-representable graphs is not closed under the lexicographical product by constructing an explicit example. 

First, observe the relation between lexicographical product and the operation of module replacing a vertex as per the following remark.

\begin{remark}\label{lexi_module}
	The lexicographical product	$G[G']$ is nothing but replacing each vertex of $G$ by the module $G'$. Hence, both $G$ and $G'$ are induced subgraphs of $G[G']$.
\end{remark}

Accordingly, in the following theorem, we provide a necessary and sufficient condition for $G[G']$ to be a word-representable graph. Thus, we settle the above-mentioned open problem.  

\begin{theorem}\label{lex_word}
	Let $G$ and $G'$ be two graphs. The lexicographical product $G[G']$ is word-representable if and only if $G$ is word-representable and $G'$ is a comparability graph. Moreover, if $\mathcal{R}(G) = k$ and $\mathcal{R}^p(G') = k'$, then $\mathcal{R}(G[G']) = \max\{k, k'\}$.
\end{theorem}

\begin{proof}
	The proof follows from Remark \ref{lexi_module}, Theorem \ref{Module_1} and Theorem \ref{Module_3}. \qed
\end{proof}

We further state a characteristic property for $G[G']$ to be a comparability graph. 

\begin{theorem}
	Let $G$ and $G'$ be two graphs. The lexicographical product $G[G']$ is a comparability graph if and only if $G$ and $G'$ are comparability graphs. Moreover, if $\mathcal{R}^p(G) = k$ and $\mathcal{R}^p(G') = k'$, then $\mathcal{R}^p(G[G']) = \max\{k, k'\}$.
\end{theorem} 

We are now ready to present the main result of the paper on  
characterization of word-representable graphs with respect to the modular decomposition, through the following lemma.  

\begin{lemma}\label{mod_word_rp}
	Let $G$ be a word-representable graph and $M$ be any non-trivial module of $G$. Then the induced subgraph $G[M]$ is a comparability graph.
\end{lemma}

\begin{proof}
	Consider the graph $G'$ which is obtained from $G$ by replacing the module $M$ by a new vertex, say $a$. Thus, $G'$ is word-representable as it is isomorphic to an induced subgraph of $G$. Note that $G$ can be reconstructed from $G'$ by replacing the vertex $a$ with the module $G[M]$. Hence, by Theorem \ref{Module_3}, we have $G[M]$ is a comparability graph. \qed
\end{proof}

\begin{theorem}\label{ch_mod_wordgraph}
	Let $G$ be a decomposable graph and $\mathcal{P} = \{M_1, M_2, \ldots, M_k\}$ be a modular partition of $G$. Then, we have the following:
	\begin{enumerate}
		\item $G$ is word-representable if and only if, for each $1 \le i \le k$, $G[M_i]$ is a comparability graph and $\sfrac{G}{\mathcal{P}}$ is a word-representable graph. 
		
		\item If $G$ is word-representable, then \[\mathcal{R}(G) = \max\{\mathcal{R}(\sfrac{G}{\mathcal{P}}), \mathcal{R}^p(G[M_1]), \ldots, \mathcal{R}^p(G[M_k])\}.\]
		
		\item If $G$ is a comparability graph, then \[\mathcal{R}^p(G) = \max\{\mathcal{R}^p(\sfrac{G}{\mathcal{P}}), \mathcal{R}^p(G[M_1]), \ldots, \mathcal{R}^p(G[M_k])\}.\]
	\end{enumerate}
\end{theorem}

\begin{proof}$\;$
	\begin{enumerate}
		\item Suppose $G$ is a word-representable graph. As $\sfrac{G}{\mathcal{P}}$ is an induced subgraph of $G$, it is an word-representable graph. Further, from Lemma \ref{mod_word_rp}, we have for each $1 \le i \le k$, $G[M_i]$ is a comparability graph. Note that the converse is evident from Remark \ref{recons_mod_dec} and Theorem \ref{Module_3}.
		
		\item If $G$ is a word-representable graph, then from the above result we have each $G[M_i]$ is a comparability graph and $\sfrac{G}{\mathcal{P}}$ is a word-representable graph. Hence, the result follows from Remark \ref{recons_mod_dec} and Theorem \ref{Module_1}.
		
		\item The result is evident from Theorem \ref{mod_com} and Theorem \ref{Module_4}. 
	\end{enumerate}
\qed
\end{proof}

\section{Concluding Remarks}
While recognizing word-representability of graphs is NP-complete \cite{Hallsorsson_2011}, it is interesting to study methods to prove the non-word-representability of a graph. To explore more in this direction, one may refer to \cite{kitaev_2024b}. In this work, we characterized word-representable graphs with respect to the modular decomposition. Further, it is known that the recognition of comparability graphs as well as the computation of modular decomposition of a graph can be done in polynomial time \cite{Golumbic_1977,McConnell_1994}. Hence, in view of Theorem \ref{ch_mod_wordgraph}, we have a polynomial time test for non-word-representability of a graph $G$: find the maximal modular partition of $G$ and check whether the modules are comparability graphs; if a non-comparability module is found, $G$ is not a word-representable graph. However, if all modules are comparability graphs, then the test gives no information. In that case, the problem is reduced to the problem of recognizing the word-representability of the associated quotient graph. Further, it is interesting to find special classes of word-representable graphs using modular decomposition.

\end{document}